\documentclass[10pt]{amsart}
\usepackage[pagebackref]{hyperref}
\hypersetup{
    colorlinks = true,
    citecolor = [rgb]{0.00, 0.5, 0.00},
}

\renewcommand*{\backref}[1]{}
\renewcommand*{\backrefalt}[4]{%
    \ifcase #1 (Not cited.)%
    \or        (Cited on page~#2.)%
    \else      (Cited on pages~#2.)%
    \fi}

\usepackage{mathtools}

\usepackage[margin=1.375in]{geometry}
\usepackage{todonotes}
\usepackage{multirow}
\usepackage{longtable}

\usepackage{amsmath}
\usepackage{amsfonts}
\usepackage{amssymb}
\usepackage{float}
\usepackage{amsthm}
\usepackage{comment}
\usepackage{amscd}
\usepackage{amsxtra}
\usepackage{epsfig}

\usepackage{listings}
\usepackage{color}

\definecolor{dkgreen}{rgb}{0,0.6,0}
\definecolor{gray}{rgb}{0.5,0.5,0.5}
\definecolor{mauve}{rgb}{0.58,0,0.82}

\usepackage{color}
\usepackage[all]{xy}
\usepackage{verbatim}
\usepackage{setspace}

\usepackage{eucal}
\usepackage{mathrsfs}

\usepackage{graphicx}
\usepackage{tikz-cd}

\usepackage{url}
\usepackage{verbatim}
\usepackage{xy}
\xyoption{all}

\usepackage{amsthm}

\usepackage{caption} 
\makeatletter
\def\@tocline#1#2#3#4#5#6#7{\relax
  \ifnum #1>\c@tocdepth % then omit
  \else
    \par \addpenalty\@secpenalty\addvspace{#2}%
    \begingroup \hyphenpenalty\@M
    \@ifempty{#4}{%
      \@tempdima\csname r@tocindent\number#1\endcsname\relax
    }{%
      \@tempdima#4\relax
    }%
    \parindent\z@ \leftskip#3\relax \advance\leftskip\@tempdima\relax
    \rightskip\@pnumwidth plus4em \parfillskip-\@pnumwidth
    #5\leavevmode\hskip-\@tempdima
      \ifcase #1
       \or\or \hskip 1em \or \hskip 2em \else \hskip 3em \fi%
      #6\nobreak\relax
    \hfill\hbox to\@pnumwidth{\@tocpagenum{#7}}\par% <---- \dotfill -> \hfill
    \nobreak
    \endgroup
  \fi}
\makeatother

\DeclareMathOperator{\Sym}{Sym}

\newtheorem{lemma}{Lemma}[section]

\newtheorem{corollary}[lemma]{Corollary}

\newtheorem{theorem}[lemma]{Theorem}

\newtheorem{prop}[lemma]{Proposition}

\theoremstyle{definition}

\newtheorem{definition}[lemma]{Definition}

\newtheorem{construction}[lemma]{Construction}
\newtheorem{remark}[lemma]{Remark}
\newtheorem{convention}[lemma]{Convention}

\newtheorem*{notation}{Notation}

\newcommand{\FF}{\mathbf{F}}
\newcommand{\Z}{\mathbf{Z}}
\newcommand{\QQ}{\mathbf{Q}}
\newcommand{\cc}{\mathbf{C}}

\newcommand{\M}{{M}}

\newcommand{\Mell}{\M_\mathrm{ell}}

\newcommand{\HP}{\mathbf{H}P}

\newcommand{\Sq}{\mathrm{Sq}}

\renewcommand{\H}{\mathrm{H}}

\newcommand{\Eoo}{{\mathbf{E}_\infty}}

\newcommand{\ol}[1]{\overline{#1}}

\newcommand{\wt}[1]{\widetilde{#1}}

\newcommand{\E}[1]{\mathbf{E}_{{#1}}}
\newcommand{\mmod}{/\!\!/}

\renewcommand{\S}{\mathbb{S}}

\newcommand{\GL}{\mathrm{GL}}

\renewcommand{\O}{\mathrm{O}}

\newcommand{\BO}{\mathrm{BO}}

\newcommand{\bo}{\mathrm{bo}}

\newcommand{\tmf}{\mathrm{tmf}}

\newcommand{\BP}[1]{\mathrm{BP}\langle{#1}\rangle}

\newcommand{\BPP}{\mathrm{BP}}
\renewcommand{\H}{\mathrm{H}}

\newcommand{\BSpin}{\mathrm{BSpin}}

\newcommand{\MSpin}{\mathrm{MSpin}}
\newcommand{\MString}{\mathrm{MString}}
\newcommand{\BString}{\mathrm{BString}}
\newcommand{\bstring}{\mathrm{bstring}}

\title{The Ando-Hopkins-Rezk orientation is surjective}
\author{Sanath Devalapurkar}
\email{sanathd@mit.edu}

\begin{document}

\maketitle

\begin{abstract}
    We show that the map $\pi_\ast \MString \to \pi_\ast \tmf$ induced by the
    Ando-Hopkins-Rezk orientation is surjective. This proves an unpublished
    claim of Hopkins and Mahowald. We do so by constructing an $\E{1}$-ring $B$
    and a map $B\to \MString$ such that the composite $B\to \MString\to \tmf$ is
    surjective on homotopy. Applications to differential topology, and in
    particular to Hirzebruch's prize question, are discussed.
\end{abstract}

\section{Introduction}

The goal of this paper is to show the following result.
\begin{theorem}\label{string-surj}
    The map $\pi_\ast \MString \to \pi_\ast \tmf$ induced by the
    Ando-Hopkins-Rezk orientation is surjective.
\end{theorem}
This integral result was originally stated as \cite[Theorem 6.25]{hopkins-icm},
but, to the best of our knowledge, no proof has appeared in the literature. In
\cite{hopkins-mahowald-orientations}, Hopkins and Mahowald give a proof sketch
of Theorem \ref{string-surj} for elements of $\pi_\ast \tmf$ of Adams-Novikov
filtration $0$.

The analogue of Theorem \ref{string-surj} for $\bo$ (namely, the statement that
the map $\pi_\ast \MSpin\to \pi_\ast \bo$ induced by the Atiyah-Bott-Shapiro
orientation is surjective) is classical \cite{milnor-spin}. In Section
\ref{abs-surj}, we present (as a warmup) a proof of this surjectivity result for
$\bo$ via a technique which generalizes to prove Theorem \ref{string-surj}. We
construct an $\E{1}$-ring $A$ with an $\E{1}$-map $A\to \MSpin$. The
$\E{1}$-ring $A$ is a particular $\E{1}$-Thom spectrum whose mod $2$ homology is
given by the polynomial subalgebra $\FF_2[\zeta_1^4]$ of the mod $2$ dual
Steenrod algebra. The Atiyah-Bott-Shapiro orientation $\MSpin\to \bo$ is an
$\Eoo$-map, and so the composite $A\to \MSpin \to \bo$ is an $\E{1}$-map.  We
then prove that the map $\pi_\ast A\to \pi_\ast \bo$ is surjective; this is
stronger than the Atiyah-Bott-Shapiro orientation being surjective on homotopy.

The argument to prove Theorem \ref{string-surj} follows the same outline: we
construct (in Section \ref{B-def}) an $\E{1}$-ring $B$ with an $\E{1}$-map $B\to
\MString$. This $\E{1}$-ring $B$ is the height $2$ analogue of the $\E{1}$-ring
$A$; this motivated the naming of $B$. We define $B$ as a particular
$\E{1}$-Thom spectrum whose mod $2$ homology is given by the polynomial
subalgebra $\FF_2[\zeta_1^8, \zeta_2^4]$ of the mod $2$ dual Steenrod algebra.
The Ando-Hopkins-Rezk orientation \cite{koandtmf} $\MString \to \tmf$ is an
$\Eoo$-map, and so the composite $B\to \MString \to \tmf$ is an $\E{1}$-map. We
then prove the following stronger statement:
\begin{theorem}\label{main-thm}
    The map $\pi_\ast B \to \pi_\ast \tmf$ is surjective.
\end{theorem}
The map $B\to \tmf$ factors through $\MString$, so Theorem \ref{string-surj}
follows. In Section \ref{invert-2}, we prove Theorem \ref{main-thm} after
localizing at $3$ (as Theorem \ref{main-thm-invert-2}). In Section
\ref{prime-2}, we prove Theorem \ref{main-thm} after localizing at $2$ (as
Theorem \ref{main-thm-prime-2}); together, these yield Theorem \ref{string-surj}
by Corollary \ref{12-equiv}. Finally, in Section \ref{apps}, we study some
applications of Theorem \ref{string-surj}. In particular, we discuss
Hirzebruch's prize question \cite[Page 86]{hirzebruch} along the lines of
\cite[Corollary 6.26]{hopkins-icm}. We also prove a conjecture of Baker's from
\cite{baker-conjecture}.

The surjectivity of the Atiyah-Bott-Shapiro orientation $\MSpin\to \bo$ was
considerably strengthened by Anderson, Brown, and Peterson in \cite{abp}: they
showed that the Atiyah-Bott-Shapiro orientation $\MSpin\to \bo$ in fact admits a
spectrum-level splitting. It is a folklore conjecture that the same is true of
the Ando-Hopkins-Rezk orientation $\MString\to \tmf$, and there have been
multiple investigations in this direction (see, for instance,
\cite{laures-k1-local, laures-k2-local}). In forthcoming work \cite{bpn-thom},
we study in detail the relationship between $B$ and $\tmf$ (as well as $A$ and
$\bo$). We show that old conjectures of Cohen, Moore, Neisendorfer, Gray, and
Mahowald in unstable homotopy theory related to the Cohen-Moore-Neisendorfer
theorem, coupled with a conjecture about the centrality of a certain element of
$\sigma_2\in \pi_{13}(B)$ (resp. $\sigma_1\in \pi_5(A)$), implies that the
Ando-Hopkins-Rezk orientation (resp. the Atiyah-Bott-Shapiro orientation) admits
a spectrum level splitting. This provides another proof of Theorem
\ref{string-surj}, assuming the truth of these conjectures.

\subsection*{Acknowledgements}

I'm extremely grateful to Mark Behrens and Peter May for agreeing to work with
me this summer and for being fantastic advisors, as well as for arranging my
stay at UChicago. I'd like to also thank Stephan Stolz for useful conversations
when I visited Notre Dame. I'm also grateful to Andrew Baker, Hood Chatham,
Jeremy Hahn, Eleanor McSpirit, and Zhouli Xu for clarifying discussions. Thanks
also to Andrew Baker, Peter May, Haynes Miller, Zhouli Xu, and in particular
Jeremy Hahn for providing many helpful comments and correcting mistakes, and to
Andrew Senger for pointing out the reference
\cite{hopkins-mahowald-orientations} after this paper was written.

\section{Warmup: surjectivity of the Atiyah-Bott-Shapiro
orientation}\label{abs-surj}

The goal of this section is to provide a proof of the following classical
theorem using techniques which generalize to prove Theorem \ref{string-surj}.
\begin{theorem}\label{spin-surj}
    The map $\pi_\ast \MSpin \to \pi_\ast \bo$ induced by the
    Atiyah-Bott-Shapiro orientation is surjective.
\end{theorem}
As mentioned in the introduction, we prove Theorem \ref{spin-surj} by
constructing an $\E{1}$-ring $A$ with an $\E{1}$-map $A\to \MSpin$. Composing
with the Atiyah-Bott-Shapiro orientation $\MSpin\to \bo$ produces an $\E{1}$-map
$A\to \MSpin \to \bo$. We then show the following result, which implies Theorem
\ref{spin-surj}. 
\begin{theorem}\label{A-surj}
    The map $\pi_\ast A\to \pi_\ast \bo$ is surjective.
\end{theorem}
\begin{remark}
    Theorem \ref{A-surj} is an old result of Mahowald's: it appears, for
    instance, as \cite[Proposition 4.1(c)]{mahowald-some-etaj} and
    \cite[Proposition 2.2(3)]{hopkins-mahowald-orientations}. Theorem
    \ref{A-surj} also implies the second part of \cite[Proposition
    4.10]{baker-characteristics}.
\end{remark}
The definition of the $\E{1}$-ring $A$ is as follows.
\begin{construction}
    Let $S^4\to \BSpin$ be a generator of $\pi_4 \BSpin \cong \Z$. Since
    $\BSpin$ is an infinite loop space, there is an induced map $\Omega S^5\to
    \BSpin$. Let $A$ denote the Thom spectrum of this map. This is an
    $\E{1}$-ring with an $\E{1}$-map $A\to \MSpin$. Its $15$-skeleton is shown
    in Figure \ref{A-15-skeleton}.
\end{construction}
\begin{remark}\label{univ-property}
    The image of a generator of $\pi_4 \BSpin$ under the J-homomorphism
    $\BSpin\to B\GL_1(\S)$ is the Hopf element $\nu\in \pi_4 B\GL_1(\S) \cong
    \pi_3 \S$. Consequently, $A$ is the Thom spectrum of the map $\Omega S^5\to
    B\GL_1(\S)$ which detects $\nu$ on the bottom cell $S^4$ of the source. In
    particular, the universal property of Thom spectra from \cite{barthel-thom}
    exhibits $A$ as the $\E{1}$-quotient $\S\mmod\nu$ of the sphere spectrum by
    $\nu$.
\end{remark}
\begin{remark}
    The spectrum $A$ is ubiquitous in Mahowald's older works
    \cite{mahowald-thom, mahowald-bo-res, mahowald-imj} (where it is often
    denoted $X_5$), where its relationship to $\bo$ via the composite $A\to
    \MSpin\to \bo$ is utilized to great effect.
\end{remark}
\begin{prop}\label{A-homology}
    The $\BPP_\ast$-algebra $\BPP_\ast(A)$ is isomorphic to a polynomial algebra
    $\BPP_\ast[y_2]$, where $|y_2| = 4$. There is a map $A_{(p)}\to \BPP$. On
    $\BPP$-homology, the element $y_2$ maps to $t_1^2$ mod decomposables at
    $p=2$.
\end{prop}
\begin{proof}
    The space $\Omega S^5$ has cells only in dimensions divisible by $4$, and
    hence the same is true of $A$. The Atiyah-Hirzebruch spectral sequence for
    $\BPP_\ast(A)$ therefore collapses at the $E^2$-page, and so $\BPP_\ast(A)
    \cong \BPP_\ast[y_2]$, as desired. Since $\pi_\ast \BPP$ is concentrated in
    even degrees, the element $\nu$ vanishes in $\pi_3 \BPP$. The universal
    property of $A$ from Remark \ref{univ-property} therefore produces an
    $\E{1}$-map $A\to \BPP$. The element $\nu$ is detected by $[t_1^2]$ in the
    $2$-local Adams-Novikov spectral sequence for the sphere (in fact, a choice
    of representative in the cobar complex is $t_1^2 + v_1 t_1$), so this yields
    the final sentence of the proposition.
\end{proof}
\begin{remark}\label{A-equiv}
    In particular, the map $A\to \bo$ is an equivalence in
    dimensions $\leq 4$.
    Proposition \ref{A-homology} implies that $\H_\ast(A; \FF_2) \cong
    \FF_2[\zeta_1^4]$; note that this is the $Q_0$-Margolis homology of
    $\H_\ast(\bo;\FF_2) \cong \FF_2[\zeta_1^4, \zeta_2^2, \zeta_3, \cdots]$.
    This is sharp: $\pi_5 A$ contains a nontrivial element $\sigma_1$ which maps
    to zero in $\pi_\ast \bo$. This element is specified up to indeterminacy by
    the relation $\eta\nu = 0$; see Figure \ref{A-15-skeleton}.
\end{remark}
\begin{figure}
    \begin{tikzpicture}[scale=0.75]
        \draw [fill] (0, 0) circle [radius=0.05];
        \draw [fill] (1, 0) circle [radius=0.05];
        \draw [fill] (1, 1) circle [radius=0.05];
        \draw [fill] (2, 0) circle [radius=0.05];
        \draw [fill] (3, 0) circle [radius=0.05];

        \draw (0,0) to node[below] {\footnotesize{$\nu$}} (1,0);
        \draw (1,0) to node[below] {\footnotesize{$2\nu$}} (2,0);
        \draw (2,0) to node[below] {\footnotesize{$3\nu$}} (3,0);

        \draw [->] (1,1) to node[left] {\footnotesize{$\eta$}} (1,0);

        \draw (0,0) to[out=-90,in=-90] node[below] {\footnotesize{$\sigma$}}
        (2,0);
    \end{tikzpicture}
    \caption{The $15$-skeleton of $A$ at the prime $2$ shown horizontally, with
    $0$-cell on the left. The element $\sigma_1$ is depicted.}
    \label{A-15-skeleton}
\end{figure}
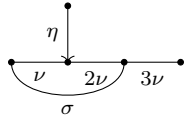
We will momentarily prove Theorem \ref{A-surj}. Before doing so, we need to
introduce one piece of notation.
\begin{notation}
    Let $M$ be a unital spectrum, and suppose $\alpha,\beta\in \pi_\ast \S$ and
    $\gamma\in \pi_\ast M$ are elements such that $\alpha\beta = 0$ and $\beta
    \gamma = 0$ in $M$. Following \cite[Section 7]{baker-may}, the Toda bracket
    $\langle \alpha, \beta, \gamma \rangle$ will denote the coset of elements of
    $\pi_{|\alpha| + |\beta| + |\gamma| + 1}(M)$ determined by the subgroup
    $\mathrm{indet} = \alpha \pi_{|\beta| + |\gamma| + 1}(M) + \pi_{|\alpha| +
    |\beta| + 1}(\S) \gamma$. The subgroup $\mathrm{indet}$ is the indeterminacy
    of the bracket $\langle \alpha, \beta, \gamma\rangle$.  There is an
    analogous definition for higher-fold Toda brackets, but we will not
    elaborate more on this, since we shall not need it. We will use this
    notation throughout without further comment.
\end{notation}
\begin{proof}[Proof of Theorem \ref{A-surj}]
    Because the map $A\to \bo$ is one of $\E{1}$-rings, it suffices to lift all
    the generators of $\pi_\ast \bo$ to $\pi_\ast A$. We first prove Theorem
    \ref{A-surj} after inverting $2$. Since $\pi_\ast \bo[1/2] \cong
    \Z[1/2][u^2]$, where $u^2$ is the square of the Bott element. It follows
    from Remark \ref{A-equiv} that the polynomial generator $u^2$ of $\pi_\ast
    \bo[1/2]$ lifts to $\pi_\ast A[1/2]$.

    It remains to prove Theorem \ref{A-surj} after $2$-localization. Recall that
    $\pi_\ast \bo$ is polynomially generated by $\eta$ in degree $1$ (which is
    spherical), $2v_1^2$ in degree $4$, and $v_1^4$ in degree $8$ (subject to
    some relations). The element $\eta\in \pi_1 \bo_{(2)}$ lifts to $\pi_1
    A_{(2)}$ because it is a spherical element and the unit $\S\to \bo_{(2)}$
    factors through $A_{(2)}$. It remains to lift the other two generators.
    Remark \ref{A-equiv} already shows that $2v_1^2$ lifts to $\pi_\ast
    A_{(2)}$. Alternatively, recall that there is a sole $d_3$-differential
    $d_3(v_1^2) = \eta^3$ in the Adams-Novikov spectral sequence for $\bo$. This
    implies that $2v_1^2\in \langle 8, \nu, 1_\bo\rangle$, with indeterminacy
    $0\pmod{2}$. Since $\nu$ vanishes in $\pi_\ast A$, we find that the bracket
    $\langle 8, \nu, 1_A\rangle$ is well-defined in $\pi_4 A_{(2)}$.

    For $v_1^4$, recall that $\sigma = 0$ in $\pi_\ast \bo$, and that $v_1^4\in
    \langle 16, \sigma, 1_\bo\rangle \subseteq \pi_\ast \bo$ (see \cite[Lemma
    7.3]{baker-may}), where the indeterminacy in this bracket is $0\pmod{2}$.
    Note that $\langle 16, \sigma, 1_\bo\rangle \subseteq \langle 4, 4\sigma,
    1_\bo\rangle$. We now observe that the attaching map of the $8$-cell of
    $A_{(2)}$ is given by $\sigma + \wt{2\nu}$, where $\wt{2\nu} \in
    \pi_7(C\nu)$ is an element determined by the relation $2\nu^2 = 0\in
    \pi_6(\S)$; see, for instance, \cite[Lemma 4.7]{baker-characteristics}. In
    particular, this implies that $4\sigma = 0$ in $\pi_\ast A_{(2)}$, so the
    bracket $\langle 4, 4\sigma, 1_A\rangle$ is well-defined.

    It follows from the above discussion that $2v_1^2$ and $v_1^4$ lift to
    $\pi_\ast A_{(2)}$ up to indeterminacy (and the indeterminacy in $\pi_\ast
    \bo$ is $0\pmod{2}$). If $2v_1^2 + 4nv_1^2 = 2(2n+1) v_1^2$ lifts to
    $\pi_\ast A_{(2)}$ for some $n\in \Z_{(2)}$, then so does $2v_1^2$ since
    $2n+1$ is a $2$-local unit.  Arguing similarly for $v_1^4$, it follows that
    the other two generators of $\pi_\ast \bo_{(2)}$ lift to $\pi_\ast A_{(2)}$,
    as desired.
\end{proof}

\section{Defining $B$}\label{B-def}

In this section, we will define the $\E{1}$-ring $B$ mentioned in the
introduction and study some of its elementary properties. It is the height $2$
analogue of the spectrum $A$ from Section \ref{abs-surj}. We define $B$ as a
Thom spectrum whose mod $2$ homology is given by $\FF_2[\zeta_1^8, \zeta_2^4]$;
notice that this is the $Q_0$-Margolis homology of the mod $2$ homology of
$\tmf$. The spectrum $B$ appeared under the name $\ol{X}$ in \cite[Section
10]{hopkins-mahowald-orientations}. We will work \emph{integrally} (i.e.,
without inverting any primes) unless explicitly mentioned otherwise.

\begin{construction}\label{B-constr}
    There is a fiber sequence
    $$S^9 = \O(10)/\O(9)\to \BO(9) \to \BO(10).$$
    There is an element $f\in \pi_{12} \O(10) \cong \Z/12$, which is sent to
    $2\nu\in\pi_{12}(S^9) \cong \Z/24$ under the boundary homomorphism in the
    long exact sequence on homotopy. Define a space $BN$ as the homotopy
    pullback
    $$\xymatrix{
	S^9 \ar[r] \ar@{=}[d] & BN\ar[r] \ar[d] & S^{13}\ar[d]^-f\\
	S^9 \ar[r] & \BO(9) \ar[r] & \BO(10).
    }$$
    Let $N$ be the loop space $\Omega BN$. If $S^9\to \mathrm{B^2 String}$
    denotes the generator of $\pi_8 \BString$, then the composite $S^{12}
    \xrightarrow{2\nu} S^9\to \mathrm{B^2 String}$ is null. The
    Atiyah-Hirzebruch-Serre spectral sequence shows that the generator of
    $\bstring^1(S^9)$ extends to $\bstring^1(BN)$, and so there is a map $BN\to
    \mathrm{B^2 String}$. The induced loop map $N\to \BString$ is given by the
    map of fiber sequences
    \begin{equation}\label{fiber-sequence-map}
	\xymatrix{
	    N \ar[r] \ar[d] & \Omega S^{13} \ar[r] \ar[d] & S^9 \ar[d]\\
	    \BString \ar[r] & \ast \ar[r] & \mathrm{B^2 String}.
	    }
    \end{equation}
    The Thom spectrum of the map $N\to \BString$ is the $\E{1}$-ring $B$.
\end{construction}

Note that $B$ is defined integrally, and that it admits an $\E{1}$-map $B\to
\MString$ obtained by Thomifying the map $N\to \BString$.

\begin{prop}\label{bp-homology}
    The $\BPP_\ast$-algebra $\BPP_\ast(B)$ is isomorphic to a polynomial algebra
    $\BPP_\ast[b_4, y_{6}]$, where $|b_4| = 8$ and $|y_6| = 12$. There is a map
    $B_{(p)}\to \BPP$. On $\BPP$-homology, the elements $b_4$ and $y_{6}$ map to
    $t_1^4$ and $t_2^2$ mod decomposables at $p=2$, and $y_{6}$ maps to $t_1^3$
    mod decomposables at $p=3$.
\end{prop}
\begin{proof}
    There is a fiber sequence
    \begin{equation}\label{fiber-sequence}
	\Omega S^9\to N\to \Omega S^{13}.
    \end{equation}
    The J-homomorphism $\BString\to B\GL_1(\S)$ gives a map $N\to \BString\to
    B\GL_1(\S)$. The composite with the map $\Omega S^9\to N$ gives a map
    $\Omega S^9\to B\GL_1(\S)$. This is the extension of the map $S^8\to
    B\GL_1(\S)$ detecting $\sigma\in \pi_7(\S)$ along $S^8\to \Omega S^9$. By
    one of the main theorems of \cite{barthel-thom}, we find that the Thom
    spectrum of the map $\Omega S^9\to B\GL_1(\S)$ is the $\E{1}$-quotient
    $\S\mmod\sigma$ of the sphere spectrum by $\sigma$.
    
    The fiber sequence \eqref{fiber-sequence} exhibits $B$ as the Thom spectrum
    of a map $\Omega S^{13}\to B\GL_1(\S\mmod\sigma)$. The induced map
    $S^{12}\to B\GL_1(\S\mmod\sigma)$ detects an element $\wt{\nu}\in
    \pi_{11}(\S\mmod\sigma)$. This element may be described as follows. The
    relation $\sigma\nu = 0$ in $\pi_\ast \S$ defines a lift of
    $\nu\in\pi_3(\S)$ to $\pi_{11}$ of the $15$-skeleton $C\sigma$ of
    $\S\mmod\sigma$; this is the element $\wt{\nu}$. Since $\BPP$ is
    concentrated in even degrees, the element $\sigma$ vanishes in $\pi_\ast
    \BPP$. Consequently, $\wt{\nu}$ is well-defined, and it, too vanishes in
    $\pi_\ast \BPP$. The universal property of Thom spectra from
    \cite{barthel-thom} then supplies an $\E{1}$-map $B\to \BPP$.

    In particular, the Thom isomorphism says that the $\BPP$-homology of $B$ is
    abstractly isomorphic as an algebra to the $\BPP$-homology of $N$. This may
    in turn be computed by the Atiyah-Hirzebruch spectral sequence. However, the
    fiber sequence \eqref{fiber-sequence} implies that the homology of $N$ is
    concentrated in even degrees. Since $\pi_\ast \BPP$ is also concentrated
    in even degrees, this implies that the Atiyah-Hirzebruch spectral sequence
    calculating $\BPP_\ast(B)$ collapses, and we find that $\BPP_\ast(B) \cong
    \BPP_\ast[b_4, y_6]$, as desired.

    The map $B\to \BPP$ induces a map $\BPP_\ast(B)\to \BPP_\ast(\BPP) \cong
    \BPP_\ast[t_1, t_2, \cdots]$. The element $\wt{\nu}$ is detected by
    $[t_2^2]$ in the $2$-local Adams-Novikov spectral sequence for
    $\S\mmod\sigma$, and by $[t_1^3]$ in the $3$-local Adams-Novikov spectral
    sequence for $\S\mmod\sigma$. The element $\sigma\in \pi_7(\S)$ is detected
    by $[t_1^4]$ in the $2$-local Adams-Novikov spectral sequence for the
    sphere. This yields the final sentence of the proposition.
\end{proof}
\begin{remark}
    Proposition \ref{bp-homology} implies that the mod $2$ homology of $B$ is
    isomorphic to $\FF_2[\zeta_1^8, \zeta_2^4]$; note that this is the
    $Q_0$-Margolis homology of $\H_\ast(\tmf;\FF_2) \cong \FF_2[\zeta_1^8,
    \zeta_2^4, \zeta_3^2, \zeta_4, \cdots]$.
\end{remark}
\begin{remark}\label{same-12}
    The composite $B\to \MString\to \tmf$ is an $\E{1}$-ring map (since the
    first map is an $\E{1}$-ring map by construction, and the second is an
    $\Eoo$-ring map by \cite{koandtmf}), and it is an equivalence in dimensions
    $\leq 12$. This follows from Proposition \ref{bp-homology}.
\end{remark}
\begin{prop}
    The map $B\to \tmf$ induces a surjection on homotopy after inverting $6$.
\end{prop}
\begin{proof}
    By \cite[Proposition 4.4]{bauer-tmf}, $\pi_\ast \tmf[1/6]$ is a polynomial
    generator on two generators $c_4$ and $c_6$, in degrees $8$ and $12$,
    respectively. Since the map $B[1/6]\to \tmf[1/6]$ is an $\E{1}$-map, the
    map $\pi_\ast B[1/6]\to \pi_\ast \tmf[1/6]$ is a ring map. It therefore
    suffices to lift the elements $c_4$ and $c_6$ to $\pi_\ast B[1/6]$. This
    follows from Remark \ref{same-12}.
\end{proof}
As an immediate consequence, we have:
\begin{corollary}\label{12-equiv}
    If the maps $\pi_\ast B_{(3)}\to \pi_\ast \tmf_{(3)}$ and $\pi_\ast
    B_{(2)}\to \pi_\ast \tmf_{(2)}$ are surjective, then Theorem \ref{main-thm}
    is true.
\end{corollary}

\begin{remark}\label{B-wood}
    In \cite{bpn-thom}, we show that $B$ is in many ways analogous to $\tmf$.
    For instance, it satisfies an analogue of the $2$-local Wood equivalence
    $\tmf_{(2)} \wedge DA_1 \simeq \tmf_1(3)_{(2)}$ from \cite{homologytmf},
    where $DA_1$ is a certain $8$-cell complex: the spectrum $B_{(2)} \wedge
    DA_1$ is a summand of Ravenel's Thom spectrum $X(4)_{(2)}$. (More precisely,
    it is the summand $T(2)$ of $X(4)_{(2)}$ obtained from the Quillen
    idempotent, as studied in \cite[Chapter 6.5]{green}.)
\end{remark}

\section{Theorem \ref{main-thm} after localizing at $3$}\label{invert-2}

In Corollary \ref{12-equiv}, we reduced Theorem \ref{main-thm} to showing that
the maps $\pi_\ast B_{(3)}\to \pi_\ast \tmf_{(3)}$ and $\pi_\ast B_{(2)}\to
\pi_\ast \tmf_{(2)}$ are surjective. Our goal in this section is to study the
$3$-local case. We shall prove:
\begin{theorem}\label{main-thm-invert-2}
    The map $\pi_\ast B_{(3)} \to \pi_\ast \tmf_{(3)}$ is surjective on
    homotopy.
\end{theorem}
\begin{convention}\label{3-localize}
    We shall localize at the prime $3$ for the remainder of this section.
\end{convention}

\subsection{The Adams-Novikov spectral sequence for $\tmf$}\label{anss-tmf}

In this section, we review the Adams-Novikov spectral sequence for $\tmf$ at
$p=3$; as mentioned in Convention \ref{3-localize}, we shall $3$-localize
everywhere. The following result is well-known, and is proved in
\cite{bauer-tmf}:
\begin{theorem}
    The $E_2$-page of the descent spectral sequence (isomorphic to the
    Adams-Novikov spectral sequence) for $\tmf$ is
    $$\H^\ast(\Mell;\omega^{\otimes 2\ast}) \cong \Z_3[\alpha, \beta, c_4, c_6,
    \Delta^{\pm 1}]/I,$$
    where $I$ is the ideal generated by the relations
    $$3\alpha = 3\beta = 0, \ \alpha^2 = 0, \ \alpha c_4 = \beta c_4 = \alpha
    c_6 = \beta c_6 = 0, \ c_4^3 - c_6^3 = 1728 \Delta.$$
    Moreover, $\alpha$ and $\beta$ are in the image of the map of spectral
    sequences from the Adams-Novikov spectral sequence of the sphere to that of
    $\tmf$, with preimages $\alpha_1$ and $\beta_1$.
\end{theorem}
The differentials are all deduced from Toda's relation $\alpha_1 \beta_1^3 = 0$
in $\pi_\ast \S$. There is a $d_5$-differential $d_5(\beta_{3/3}) = \alpha_1
\beta_1^3$ (the ``Toda differential''), where $\beta_{3/3}$ lives in bidegree
$(t-s,s) = (34,2)$; see, e.g., \cite[Theorem 4.4.22]{green}. Under the
$\Eoo$-ring map $\S\to \tmf$, this pushes forward to the same differential in
the Adams-Novikov spectral sequence for $\tmf$. Then:
\begin{lemma}\label{beta-3-3}
    There is a relation $\beta_{3/3} = \Delta \beta$ in the $E_2$-page of the
    Adams-Novikov spectral sequence for $\tmf$.
\end{lemma}
\begin{proof}
    We explain how to deduce this from the literature. Multiplication by
    $\alpha$ is an isomorphism in the Adams-Novikov spectral sequence for both
    the sphere and $\tmf$ in stem $34$, so it suffices to check that $\alpha
    \beta_{3/3} = \Delta \alpha\beta$. The class $\alpha_1 \beta_{3/3}$
    (resp. $\Delta \alpha \beta$) is a permanent cycle in the Adams-Novikov
    spectral sequence of the sphere (resp. $\tmf$) by the discussion on
    \cite[Page 137]{green}. It is known (see \cite[Chapter 13, page 12]{tmf})
    that $\Delta \alpha \beta$ detects $\alpha_1 \beta_{3/3}$ in homotopy. To
    conclude that they are the same on the $E_2$-page of the Adams-Novikov
    spectral sequence for $\tmf$, it suffices to note that $\alpha_1
    \beta_{3/3}$ maps to (a unit multiple of) $\Delta \alpha\beta$, as desired.
\end{proof}
It follows by naturality that there is a $d_5$-differential $d_5(\Delta \beta) =
\alpha \beta^3$, which gives (by $\beta$-linearity):
\begin{prop}\label{d5-tmf-3}
    In the Adams-Novikov spectral sequence for $\tmf$, there is a
    $d_5$-differential $d_5(\Delta) = \alpha \beta^2$.
\end{prop}
Since $3\alpha = 0$ in the Adams-Novikov spectral sequence of $\tmf$, we must
have $d_5(3\Delta) = 3\alpha\beta^2 = 0$. There are no other possibilities for
differentials on $3\Delta$, so it is a permanent cycle. Proposition
\ref{d5-tmf-3} shows that there is a Toda bracket $3\Delta\in
\langle 3, \alpha, \beta^2\rangle$ in $\pi_\ast \tmf$. This can be expressed by
the claim that $3\Delta$ can be expressed a composite
$$S^{24} \to \Sigma^{20} C\alpha_1\xrightarrow{\beta^2} \tmf,$$
where the first map is of degree $3$ on the top cell.

By $\Delta$-linearity, there is also a $d_5$-differential $d_5(\Delta^2) =
\alpha \beta^2 \Delta$, so $3\Delta^2$ lives in the $E_6$-page. There are no
further possibilities for differentials, so $3\Delta^2$ lives in $\pi_\ast
\tmf$. Again, this shows that $3\Delta^2\in \langle 3, \Delta \alpha,
\beta^2\rangle$.
Finally, we turn to $\Delta^3$. We have $d_5(\Delta^3) = 3\Delta^2 \alpha
\beta^2$, so we find that $\Delta^3 \in \langle 3, \Delta^2 \alpha,
\beta^2\rangle$.
We collect our conclusions in the following:
\begin{corollary}\label{bracket-delta}
    The following is true in $\pi_\ast \tmf$:
    \begin{itemize}
	\item $3\Delta\in \langle 3, \alpha, \beta^2\rangle$;
	\item $3\Delta^2 \in \langle 3, \Delta\alpha, \beta^2\rangle$
	    %\subseteq \langle 3, \alpha, \alpha, \beta^4\rangle$;
	\item $\Delta^3 \in \langle 3, \Delta^2\alpha, \beta^2\rangle$.
    \end{itemize}
\end{corollary}
\begin{remark}\label{toda-unique-3}
    The indeterminacy of the above Toda brackets in $\pi_\ast \tmf_{(3)}$ are
    $3\Z_{(3)}\{3\Delta\}$, $3\Z_{(3)}\{3\Delta^2\}$, and
    $3\Z_{(3)}\{\Delta^3\}$, respectively.
\end{remark}

\subsection{The Adams-Novikov spectral sequence for $B$}

In this section, we analyze the ring map $B\to \tmf$, and show that the
generators of $\pi_\ast \tmf_{(3)}$ lift to $\pi_\ast B_{(3)}$. By Corollary
\ref{12-equiv}, this implies Theorem \ref{main-thm-invert-2}. We begin by
showing:
\begin{prop}\label{delta-lift}
    There is an element in the $E_2$-page of the Adams-Novikov spectral sequence
    for $B$ which lifts the element $\Delta$ in the $E_2$-page of the
    Adams-Novikov spectral sequence for $\tmf$.
\end{prop}
\begin{proof}
    To prove the proposition, we begin by recalling the definition of a
    representative for the element $\Delta$ in the cobar complex computing the
    $E_2$-page of the Adams-Novikov spectral sequence for $\tmf$. The Hopf
    algebroid $(\BPP_\ast \tmf, \BPP_\ast \BPP\otimes_{\BPP_\ast} \BPP_\ast
    \tmf)$ is isomorphic to the elliptic curve Hopf algebroid $(A, \Gamma)$
    presenting the moduli stack of cubic curves by \cite[Corollary
    5.3]{homologytmf}. Recall from \cite[Page 16]{bauer-tmf} (or \cite[Section
    III.1]{silverman}) that for an elliptic curve in Weierstrass form
    \begin{equation}\label{weier}
	y^2 + a_1 xy + a_3 y = x^3 + a_2 x^2 + a_4 x + a_6,
    \end{equation}
    we can define quantities
    $$b_2 = a_1^2 + 4a_2, \ b_4 = 2a_4 + a_1 a_3, \ b_6 = a_3^2+ 4a_6, \ b_8 =
    a_1^2 a_6 + 4a_2 a_6 - a_1 a_3 a_4 + a_2 a_3^2 - a_4^2,$$
    which allows us to define elements
    $$c_4 = b_2^2 - 24 b_4, \ c_6 = -b_2^3 + 36 b_2 b_4 - 216 b_6.$$
    The discriminant is
    $$\Delta = -b_2^2 b_8 - 8 b_4^3- 27 b_6^2 + 9b_2 b_4 b_6.$$
    Now, it is known that upon inverting $2$, every elliptic curve in
    Weierstrass form \eqref{weier} is isomorphic to one of the form
    \begin{equation}\label{weier-3}
	y^2 = x^3 + a_2 x^2 + a_4 x.
    \end{equation}
    It follows that the elliptic curve Hopf algebroid is isomorphic to a Hopf
    algebroid of the form $(A',\Gamma') = (\Z[1/2][a_2, a_4], A'[r]/(r^3 + a_2
    r^2 + a_4 r))$, where $I$ is some ideal consisting of complicated
    relations, and where the Hopf algebroid structure can be written down
    explicitly (as in \cite[Section 3]{bauer-tmf}). A straightforward
    calculation proves that the discriminant is then 
    \begin{equation}\label{Delta-3}
	\Delta = a_2^2 b_4^2 - 16 b_4^3.
    \end{equation}
    Turning to $B$, recall that $\BPP_\ast B \cong \BPP_\ast[b_4, y_{6}]$. The
    map $(\BPP_\ast B, \BPP_\ast \BPP \otimes_{\BPP_\ast} \BPP_\ast B) \to (A',
    \Gamma')$ of Hopf algebroids induced by the map $B\to \tmf$ sends $b_4$ to
    $b_4$ and $y_{6}$ to $a_2 b_4$ mod decomposables. It follows from Equation
    \eqref{Delta-3} that the element $\Delta$ already exists in the $0$-line of
    the Adams-Novikov spectral sequence for $B$.
    %the discriminant is represented by $$\Delta = [y_{6}^2 - 16 b_4^3]\in
    %\Ext(\BPP_\ast B).$$
    Using Sage to calculate the $3$-series of the formal group law of the
    elliptic curve \eqref{weier-3}, one finds that $v_1$ is $a_2$ up to a
    $3$-adic unit. We conclude that
    $$c_4 = 4v_1^2 - 24b_4, \ c_6 = -4v_1^3 - 144 y_{6}.$$
    This completes the proof of Proposition \ref{delta-lift}.
\end{proof}

By Remark \ref{same-12}, the elements $c_4,c_6\in \pi_\ast \tmf$ lift to
$\pi_\ast B$. The key to lifting the other elements of $\pi_\ast \tmf$ is the
following:
\begin{theorem}\label{d5-B}
    There is a differential $d_5(\Delta) = \alpha \beta^2$ in the Adams-Novikov
    spectral sequence for $B$. Moreover, $\alpha \beta^2$ vanishes in $\pi_\ast
    B$, and $3\Delta$ is a permanent cycle.
\end{theorem}
\begin{proof}
    The element $\alpha \beta^2$ is detected in filtration $5$ in the
    Adams-Novikov spectral sequence for the sphere. We first check that there is
    no class above filtration $5$ in stem $23$ the Adams-Novikov spectral
    sequence for $B$. In Figure \ref{B-cell}, we depict the $20$-skeleton of
    $B$. Now, $\alpha \beta^2$ is the first class in filtration $5$ in the
    Adams-Novikov spectral sequence for the sphere, so there are no classes
    above filtration $5$ in stem $23$ in the algebraic Atiyah-Hirzebruch
    spectral sequence (converging to the Adams-Novikov spectral sequence of
    $B$). Consequently, there are no classes above filtration $5$ in stem $23$
    of the Adams-Novikov spectral sequence for $B$. It follows that $\alpha
    \beta^2$ must be detected in filtration $5$ in the Adams-Novikov spectral
    sequence for $B$. Moreover, if the $d_5$-differential on $\Delta$ exists,
    then it is the longest one (and hence $3\Delta$ is a permanent cycle).

    We now prove the $d_5$-differential. We first claim that there is no nonzero
    target for a $d_r$-differential on $\Delta$ for $2\leq r\leq 4$.  Indeed,
    such a class must live in bidegree $(t-s,s) = (23,r)$, so we only need to
    check that there are no classes in that bidegree. Such a class can
    only possibly come from those permanent cycles in the algebraic
    Atiyah-Hirzebruch spectral sequence which are supported on stems $23-8 =
    15$, $23-12 = 11$, $23-16 = 7$, or $23-20 = 3$ of the Adams-Novikov spectral
    sequence of the sphere.  The only classes in these stems are in
    Adams-Novikov filtration $1$, so cannot possibly contribute to a class that
    lives in bidegree $(t-s,s) = (23,r)$ with $2\leq r\leq 4$. Therefore, the
    first possibility for a differential on $\Delta$ is the $d_5$-differential
    $d_5(\Delta) = \alpha \beta^2$. The existence of this differential is forced
    by the same differential in the Adams-Novikov spectral sequence for $\tmf$.

    Therefore, $\alpha\beta^2$ vanishes in the $E_\infty$-page of the ANSS for
    $B$; there may, however, be a multiplicative extension causing $\alpha
    \beta^2$ to be nonzero in $\pi_\ast B$. But multiplicative extensions have
    to jump filtration, and we established that there are no classes above
    filtration $5$ in stem $23$ of the Adams-Novikov spectral sequence for $B$.
    Therefore, $\alpha\beta^2 = 0$ in $\pi_\ast B$, as desired.

    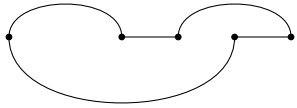
\begin{figure}
        \begin{tikzpicture}[scale=0.75]
	    \draw [fill] (0, 0) circle [radius=0.05];
	    \draw [fill] (2, 0) circle [radius=0.05];
	    \draw [fill] (3, 0) circle [radius=0.05];
	    \draw [fill] (4, 0) circle [radius=0.05];
	    \draw [fill] (5, 0) circle [radius=0.05];

	    \draw (2,0) to (3,0);
	    \draw (4,0) to (5,0);

	    \draw (0,0) to[out=90,in=90] (2,0);
	    \draw (3,0) to[out=90,in=90] (5,0);

	    \draw (0,0) to[out=-90,in=-90] (4,0);
	\end{tikzpicture}
	\caption{Cell structure of the $20$-skeleton of $B$; the bottom cell (in
	dimension $0$) is on the left; straight lines are $\alpha_1$, and curved
	lines correspond to $\alpha_2$ and $\alpha_4$, in order of increasing
	length.}
        \label{B-cell}
    \end{figure}
\end{proof}
\begin{corollary}
    The elements $3\Delta,3\Delta^2,\Delta^3\in\pi_\ast \tmf$ lift to $\pi_\ast
    B$.
\end{corollary}
\begin{proof}
    Theorem \ref{d5-B} verifies that $3\Delta$ lifts to $\pi_\ast B$ and that
    the brackets in Corollary \ref{bracket-delta} are well-defined in $\pi_\ast
    B$. This implies that $3\Delta^2$ and $\Delta^3$ in $\pi_\ast \tmf$ lift to
    $\pi_\ast B$ up to indeterminacy. Remark \ref{toda-unique-3} tells us the
    indeterminacy of the brackets in Corollary \ref{bracket-delta}. If
    $3\Delta^2 + 3n[3\Delta^2] = 3(3n+1)\Delta^2$ (resp. $\Delta^3 + 3n\Delta^3
    = (3n+1)\Delta^3$) lifts for some nonzero $n\in \Z_{(3)}$, then so does
    $3\Delta^2$ (resp. $\Delta^3$) since $3n+1$ is a $3$-local unit.
\end{proof}
The elements $\alpha$, $\beta$, $c_4$, $c_6$, $3\Delta$, $3\Delta^2$,
$\Delta^3$, and $b = \langle \beta^2, \alpha, \alpha\rangle$ (no indeterminacy)
generate the homotopy of $\tmf$. Moreover, $\alpha \beta^2 = 0$ in $\pi_\ast B$
and $\alpha^2 = 0$ in the sphere, so $b$ admits a lift to $\pi_\ast B$.
Therefore, all generators of $\pi_\ast \tmf$ admit lifts to $\pi_\ast B$; this
yields Theorem \ref{main-thm-invert-2}.

\section{Theorem \ref{main-thm} after localizing at $2$}\label{prime-2}

Our goal in this section is to prove:
\begin{theorem}\label{main-thm-prime-2}
    The map $\pi_\ast B_{(2)} \to \pi_\ast \tmf_{(2)}$ is surjective on
    homotopy.
\end{theorem}
Together with Theorem \ref{main-thm-invert-2} and Corollary \ref{12-equiv}, this
proves Theorem \ref{main-thm}.
\begin{convention}
    We shall localize at $2$ throughout this section, unless explicitly
    mentioned otherwise.
\end{convention}

\subsection{The Adams-Novikov spectral sequence for $\tmf$}\label{anss-tmf-2}

In this section, we review the Adams-Novikov spectral sequence for $\tmf$ at
$p=2$. The following result is well-known, and is proved in \cite{bauer-tmf}
(see also \cite[Proposition 1.4.9]{mark-handbook}):
\begin{theorem}
    The $E_2$-page of the descent spectral sequence (isomorphic to the
    Adams-Novikov spectral sequence) for $\tmf$ is
    $$\H^\ast(\Mell; \omega^{2\ast}) \cong \Z_{(2)}[c_4, c_6, \Delta^{\pm 1},
    \eta, a_1^2 \eta, \nu, \epsilon, \kappa, \ol{\kappa}]/I,$$
    where $I$ is the ideal generated by the relations
    \begin{gather*}
	2\eta, \eta \nu, 4\nu, 2\nu^2, \nu^3 = \eta\epsilon, \\
	2\epsilon, \nu\epsilon, \epsilon^2, 2a_1^2 \eta, \nu a_1^2 \eta,
	\epsilon a_1^2 \eta, (a_1^2 \eta)^2 = c_4 \eta^2, \\
	2\kappa, \eta^2 \kappa, \nu^2 \kappa = 4\ol{\kappa}, \epsilon\kappa,
	\kappa^2, \kappa a_1^2 \eta, \\
	\nu c_4, \nu c_6, \epsilon c_4, \epsilon c_6, a_1^2 \eta c_4 = \eta c_6,
	a_1^2 \eta c_6 = \eta c_4^2, \\
	\kappa c_4, \kappa c_6, \ol{\kappa} c_4 = \eta^4 \Delta, \ol{\kappa} c_6
	= \eta^2 (a_1^2 \eta) \Delta, 1728 \Delta = c_4^3 - c_6^2.
    \end{gather*}
\end{theorem}

\begin{remark}\label{c4-2c6}
    The elements $c_4$ and $2c_6$ are permanent cycles. There is a map $\tmf \to
    \tmf_1(3)$, where the target is complex oriented. The elements $c_4$ and
    $2c_6$ are nontrivial in $\pi_\ast \tmf_1(3)$. In fact, the image of the map
    $\tmf\to \tmf_1(3)$ consists of the elements $c_4$, $2c_6$, $c_4 \Delta^k$,
    and $2c_6 \Delta^k$ for $k\geq 1$, so these elements must be permanent
    cycles in the Adams-Novikov spectral sequence for $\tmf$.
\end{remark}

The ANSS for $\tmf$ is essentially determined from Toda's relation $\ol{\kappa}
\nu^3 = 0$ in $\pi_{29} \S$. We will explain this statement in the rest of this
section. The relation $\ol{\kappa} \nu^3 = 0\in \pi_{29} \S$ is enforced by the
differential $d_5(\beta_{6/2}) = \ol{\kappa} \nu^3$ in the ANSS for the sphere
(see \cite{isaksen-anss-charts}).  Then:
\begin{lemma}
    There is a relation $\beta_{6/2} = \Delta \nu^2$ in the $E_2$-page of the
    Adams-Novikov spectral sequence for $\tmf$.
\end{lemma}
This gives the differential $d_5(\Delta \nu^2) = \ol{\kappa} \nu^3$ in the ANSS
for $\tmf$. By $\nu$-linearity, we have $d_5(\Delta) = \ol{\kappa}\nu$. Since
$4\nu = 0$ in the $E_2$-page of the ANSS, the class $4\Delta$ survives. The
relation $4\nu = \eta^3$ forces a $d_7$-differential on $4\Delta$. In summary:
\begin{theorem}\label{d5}
    There are differentials $d_5(\Delta) = \ol{\kappa}\nu$ and $d_7(4\Delta) =
    \ol{\kappa} \eta^3$ in the ANSS for $\tmf$, and $\ol{\kappa}\nu = 0$ in
    $\pi_\ast \tmf$.
\end{theorem}
In particular, since $2\eta = 0$ in the ANSS, $8\Delta$ survives to the
$E_8$-page. There are no more differentials, so it is a permanent cycle.
Theorem \ref{d5} then shows that there is a Toda bracket $8\Delta \in \langle 8,
\nu, \ol{\kappa}\rangle$ in $\pi_\ast \tmf$; this bracket is well-defined since
$8\nu = 0$ in $\pi_\ast \S$. This can be expressed by the claim that $8\Delta$
may be expressed as a composite
$$S^{24} \to \Sigma^{20} C\nu \xrightarrow{\ol{\kappa}} \tmf,$$
where the first map is degree $8$ on the top cell. Similarly, $\Delta \eta \in
\langle \eta, \nu, \ol{\kappa}\rangle$ in $\pi_\ast \tmf$; this bracket is
well-defined since $\eta\nu = 0$ in $\pi_\ast \S$. Arguing in the same way, and
using the spherical relations $2\nu^2 = 0$, $\epsilon\nu = 0$, we find that:
\begin{prop}\label{toda-1}
    The following Toda brackets exist in $\pi_\ast \tmf$:
    \begin{enumerate}
	\item $8\Delta \in \langle 8, \nu, \ol{\kappa}\rangle$;
	\item $\Delta\eta = \langle \eta, \nu, \ol{\kappa}\rangle$;
	\item $2\Delta\nu = \langle 2\nu, \nu, \ol{\kappa}\rangle$;
	\item $\Delta\epsilon = \langle \epsilon, \nu, \ol{\kappa}\rangle$;
	    %$\Z/4\{2\ol{\kappa} c_6 = 0\}$;
	\item $\Delta\eta \kappa = \langle \eta\kappa, \nu,
	    \ol{\kappa}\rangle$;
	    %$\Z/2\{c_4^2 \eta^2 \ol{\kappa} = 0\}$;
	\item $\Delta\eta\ol{\kappa} = \langle \eta\ol{\kappa}, \nu,
	    \ol{\kappa}\rangle$.
    \end{enumerate}
    None of these except the first have any indeterminacy.
\end{prop}
To describe the other elements in $\pi_\ast \tmf$, we adopt a slightly different
approach from Section \ref{anss-tmf} --- we will not bother writing down all the
generators of $\pi_\ast \tmf$ as Toda brackets of spherical elements unless it
is convenient/necessary to do so. This is only to streamline exposition,
although one can of course work this out at one's own leisure; see Remark
\ref{james}.

The $d_5$-differential on $\Delta$ forces a differential $d_5(\Delta^k) =
k\Delta^{k-1} \ol{\kappa}\nu$. The $d_7$-differential $d_7(\Delta^4) =
\Delta^3\ol{\kappa}\eta^3$ now implies that the classes $\{\Delta^{8k},
2\Delta^{8k+4}, 4\Delta^{4k+2}, 8\Delta^{2k+1}\}$ survive to the $E_8 =
E_9$-page. In fact, these are permanent cycles. A simple induction on $k$
shows:
\begin{prop}\label{toda-2}
    Up to units, we have
    \begin{enumerate}
	\item $\Delta^{8k}\in \langle 2, \Delta^{8k-1} \eta^3,
	    \ol{\kappa}\rangle$ with indeterminacy $2\Z_{(2)}\{\Delta^{8k}\}$;
	\item $2\Delta^{8k+4} \in \langle 2, \Delta^{8k+3} \eta^3,
	    \ol{\kappa}\rangle$ with indeterminacy $2\Z_{(2)}\{2\Delta^{8k+4}\}$;
	\item $4\Delta^{4k+2} \in \langle 2, 2\Delta^{4k+1} \nu,
	    \ol{\kappa}\rangle$ with indeterminacy $2\Z_{(2)}\{4\Delta^{4k+2}\}$;
	\item $8\Delta^{2k+1} \in \langle 8, \Delta^{2k} \nu,
	    \ol{\kappa}\rangle$ with indeterminacy $8\Z_{(2)}\{8\Delta^{2k+1}\}$.
    \end{enumerate}
\end{prop}

We now turn to the other generators of $\pi_\ast \tmf$, listed in \cite[Figure
1.2]{mark-handbook}.
\begin{prop}\label{toda-3}
    We have the following Toda brackets in $\pi_\ast \tmf$, each without any
    indeterminacy:
    \begin{enumerate}
	\item $\Delta^2 \nu = \langle \nu, 2\nu \Delta, \ol{\kappa}\rangle$;
	\item $\Delta^4 \eta = \langle \eta, \Delta^3 \eta^3,
	    \ol{\kappa}\rangle$;
	\item $\Delta^4 \nu = \langle \nu, \Delta^3 \eta^3,
	    \ol{\kappa}\rangle$;
	\item $\Delta^4 \epsilon = \langle \epsilon, \Delta^3 \eta^3,
	    \ol{\kappa}\rangle$;
	\item $\Delta^4 \kappa = \langle \kappa, 4\nu, 3\nu, 2\nu, \nu,
	    \ol{\kappa}^4\rangle$;
	\item $2\Delta^5 \nu = \langle 2\nu, \Delta^4 \nu, \ol{\kappa}\rangle$;
	\item $\Delta^5 \epsilon = \langle \epsilon, \Delta^4 \nu,
	    \ol{\kappa}\rangle$;
	\item $\Delta^6\nu = \langle \nu, 2\Delta^5\nu, \ol{\kappa}\rangle$;
    \end{enumerate}
\end{prop}
\begin{remark}\label{products}
    We have excluded those elements which can be derived using the
    multiplicative structure. All other elements (except for $c_4 \Delta^k$ and
    $2c_6 \Delta^k$) can be expressed as products of the elements listed in
    Propositions \ref{toda-1}, \ref{toda-2}, and \ref{toda-3}. Importantly, the
    proofs of these propositions \emph{only} use $\ol{\kappa}\nu = 0$ in
    $\pi_\ast \tmf$ (via Theorem \ref{d5}) and multiplicative relations in
    the sphere.
\end{remark}
\begin{remark}\label{james}
    There are a lot of interesting multiplicative extensions, described in
    \cite[Section 8]{bauer-tmf}, but we will not need them.  Each of these
    relations can be derived essentially only using the $d_5$-differential of
    Theorem \ref{d5} and the multiplicative structure in the homotopy of the
    sphere.

    We can recast these extensions from the following perspective. The spectrum
    $C\nu$ is the Thom spectrum of the Spin-bundle over $S^4$ determined by the
    generator of $\pi_4 \BSpin$. Since $\BSpin$ is an infinite loop space, this
    bundle extends to one over $\Omega S^5$, and hence over the intermediate
    James constructions $J_k(S^4)$ for all $k\geq 1$. Let $J_k(S^4)^\mu$ denote
    the Thom spectrum of this bundle, so $J_1(S^4)^\mu = C\nu$. Since
    $\{J_k(S^4)\}$ forms a filtered $\E{1}$-space, we obtain a map $C\nu^{\wedge
    k} \to J_k(S^4)^\mu$. Taking the product of $\ol{\kappa}:\Sigma^{20} C\nu\to
    \tmf$ with itself $k$ times defines a map
    $$\ol{\kappa}^k: \Sigma^{20k} J_k(S^4)^\mu \to \Sigma^{20k} J_k(S^4)^\mu
    \wedge \tmf \to \tmf.$$
    Suppose $x\in \pi_\ast \S$ lifts to a map $S^{4k + |x|}\to J_k(S^4)^\mu$
    which given by $x$ on the top ($4k$-dimensional) cell of $J_k(S^4)^\mu$.
    Then the composite $S^{4k + |x|}\to J_k(S^4)^\mu \wedge \tmf$ defines an
    element of the form $x \Delta^k\in \pi_{24k + |x|} \tmf$.  For instance, we
    have:
    \begin{enumerate}
	\item $\Delta^2 \nu\in \langle \nu, 2\nu, \nu, \ol{\kappa}^2\rangle$;
	\item $\Delta^4 \eta\in \langle \eta, 4\nu, 3\nu, 2\nu, \nu,
	    \ol{\kappa}^4\rangle$;
	\item $\Delta^4\nu\in \langle \nu, 4\nu, 3\nu, 2\nu, \nu,
	    \ol{\kappa}^4\rangle$;
	\item $\Delta^4 \epsilon\in \langle \epsilon, 4\nu, 3\nu, 2\nu, \nu,
	    \ol{\kappa}^4\rangle$;
	\item $\Delta^4 \kappa\in \langle \kappa, 4\nu, 3\nu, 2\nu, \nu,
	    \ol{\kappa}^4\rangle$;
	\item $2\Delta^5 \nu\in \langle 2\nu, 5\nu, 4\nu, 3\nu, 2\nu, \nu,
	    \ol{\kappa}^5\rangle$;
	\item $\Delta^5 \epsilon\in \langle \epsilon, 5\nu, 4\nu, 3\nu, 2\nu,
	    \nu, \ol{\kappa}^5\rangle$.
    \end{enumerate}
    The brackets in (b), (c), and (e) appear in \cite[Corollary 8.7]{bauer-tmf}.
    The others may also be obtained by arguing as Bauer does: they are
    consequences of the bracket $\ol{\kappa} = \langle \nu, 2\nu, 3\nu, 4\nu,
    \nu, \eta\rangle = \langle \nu, 2\nu, 3\nu, 4\nu, \eta, \nu\rangle$ in
    $\pi_\ast \tmf$ (no indeterminacy), stated as \cite[Lemma 8.6]{bauer-tmf}.
\end{remark}
\begin{remark}
    Mark Behrens pointed out to us that Mahowald expected $\ol{\kappa}^7 = 0$ in
    $\pi_\ast \S_{(2)}$ (it is known that $\ol{\kappa}^6 = 0$ in $\pi_\ast
    \tmf_{(2)}$). It would be interesting to know whether this is related to the
    existence of $\Delta^8$ in $\pi_\ast \tmf$ via the approach given in Remark
    \ref{james}.
\end{remark}
Finally, we prove Proposition \ref{toda-3}.
\begin{proof}[Proof of Proposition \ref{toda-3}]
    We prove this case-by-case.
    \begin{enumerate}
	\item Since $d_5(\Delta^2) = 2\Delta \ol{\kappa}\nu$ and $2\nu^2 = 0$ in
	    the ANSS for the sphere, we find that $\Delta^2 \nu\in \langle \nu,
	    2\nu \Delta, \ol{\kappa}\rangle$. We provide the argument for
	    indeterminacy in this case, but not for the others since the
	    argument is essentially the same. The indeterminacy lives in
	    $\ol{\kappa}\pi_{31} \tmf + \nu \pi_{48} \tmf$, but
	    $\ol{\kappa}\pi_{31} \tmf \cong \nu \pi_{48} \tmf \cong 0$.
	\item Since $d_7(\Delta^4) = \Delta^3 \ol{\kappa}\eta^3$, we have
	    $d_7(\Delta^4 \eta) = \Delta^3 \ol{\kappa}\eta^4 = 0$. Therefore,
	    $\Delta^4 \eta\in \langle \eta, \Delta^3 \eta^3,
	    \ol{\kappa}\rangle$. This bracket is well-defined because $\Delta^3
	    \eta^3 = 4\Delta^3 \nu$ exists in $\pi_\ast \tmf$, $\eta \nu = 0$ in
	    the sphere, and $\ol{\kappa}\eta^3 = 0$ in $\tmf$.
	\item Similarly, since $d_7(\Delta^4) = \Delta^3 \ol{\kappa}\eta^3$, we
	    have $d_7(\Delta^4\nu) = \Delta^3 \ol{\kappa}\eta^3 \nu = 0$.
	    Therefore $\Delta^4 \nu\in \langle \nu, \Delta^3 \eta^3,
	    \ol{\kappa}\rangle$. This bracket is well-defined because $\Delta^3
	    \eta^3$ exists in $\pi_\ast \tmf$, $\eta\nu = 0$ in the sphere, and
	    $\ol{\kappa}\eta^3$ vanishes in $\tmf$.
	\item Similarly, since $d_7(\Delta^4) = \Delta^3 \ol{\kappa}\eta^3$, we
	    have $d_7(\Delta^4\epsilon) = \Delta^3 \ol{\kappa}\eta^3\epsilon =
	    0$, since $2\epsilon = 0$. Therefore, $\Delta^4 \epsilon\in \langle
	    \epsilon, \Delta^3 \eta^3, \ol{\kappa}\rangle$. This bracket is
	    again well-defined.
	\item This is in \cite[Corollary 8.7]{bauer-tmf}, where $\Delta^4
	    \kappa$ is denoted $e[110,2]$.
	\item Since $d_5(\Delta^5) = 5\Delta^4 \ol{\kappa}\nu$, we have
	    $d_5(2\Delta^5\nu) = 10\Delta^4 \ol{\kappa}\nu^2 = 0$, since $2\nu^2
	    = 0$. It follows that $2\Delta^5 \nu \in 5\langle 2\nu,
	    \Delta^4 \nu, \ol{\kappa}\rangle$. This is well-defined because
	    $\Delta^4\nu$ lives in $\pi_\ast \tmf$, $2\nu^2 = 0$ in the sphere,
	    and $\ol{\kappa}\nu = 0$ in $\tmf$.
	\item Similarly, since $d_5(\Delta^5) = 5\Delta^4 \ol{\kappa}\nu$, we
	    have $d_5(\Delta^5\epsilon) = 5\Delta^4 \ol{\kappa}\nu\epsilon = 0$,
	    because $\epsilon\nu = 0$. It follows that $\Delta^5 \epsilon\in 5
	    \langle \epsilon, \Delta^4 \nu, \ol{\kappa}\rangle$, which is
	    well-defined because $\Delta^4 \nu$ lives in $\pi_\ast \tmf$,
	    $\epsilon\nu = 0$ in the sphere, and $\ol{\kappa}\nu = 0$ in $\tmf$.
	\item Since $d_5(\Delta^6) = 6\Delta^5 \ol{\kappa}\nu$, we have
	    $d_5(\Delta^6\nu) = 6\Delta^5 \ol{\kappa}\nu^2 = 0$. We therefore
	    have $\Delta^6\nu \in 3\langle \nu, 2\Delta^5 \nu,
	    \ol{\kappa}\rangle$. This is well-defined because $2\nu \Delta^5$
	    lives in $\pi_\ast \tmf$, $2\nu^2 = 0$ in the sphere, and
	    $\ol{\kappa}\nu = 0$ in $\tmf$.
    \end{enumerate}
\end{proof}

\subsection{The Adams-Novikov spectral sequence for $B$}\label{anss-B}

In this section, we analyze the ring map $B\to \tmf$, and show that the
generators of $\pi_\ast \tmf_{(2)}$ lift to $\pi_\ast B_{(2)}$. Again, we will
localize at $p=2$ throughout.

We begin by showing:
\begin{prop}\label{Delta}
    There is an element in the $0$-line of the $E_2$-page of the ANSS for $B$
    which lifts the element $\Delta$ in the $E_2$-page of the ANSS for $\tmf$.
\end{prop}
\begin{proof}
    We begin by recalling a representative for $\Delta$ in the cobar complex for
    $\tmf$ at $p=2$. Recall from Proposition \ref{delta-lift} that the Hopf
    algebroid $(\BPP_\ast \tmf, \BPP_\ast \BPP\otimes_{\BPP_\ast} \BPP_\ast
    \tmf)$ is isomorphic to the elliptic curve Hopf algebroid $(A, \Gamma)$
    presenting the moduli stack of cubic curves. As in the $3$-complete setting
    (studied in Proposition \ref{delta-lift}), it is known that upon
    $2$-completion, every elliptic curve in Weierstrass form is isomorphic to
    one of the form
    $$y^2 + a_1 xy + a_3 y = x^3.$$
    Consequently (as in the $3$-complete setting), the elliptic curve Hopf
    algebroid is isomorphic to a Hopf algebroid of the form $(A', \Gamma') =
    (\Z_2[a_1, a_3], A'[s,t]/I)$, where $I$ is some ideal consisting of
    complicated relations, and where the Hopf algebroid structure can be written
    down explicitly (as in \cite[Section 3]{bauer-tmf}). A straightforward
    calculation proves that the discriminant is then
    \begin{equation}\label{Delta-2}
        \Delta = a_1^3 a_3^3 - 27 a_3^4 = b_4^3 - 27 b_6^2.
    \end{equation}
    Turning to $B$, recall that we may identify $\BPP_\ast B$ with
    $\BPP_\ast[b_4, y_{6}]$. The map $B\to \tmf$ induces a map $(\BPP_\ast B,
    \BPP_\ast \BPP \otimes_{\BPP_\ast} \BPP_\ast B) \to (A', \Gamma')$ of Hopf
    algebroids that sends $b_4$ to $b_4$ and $y_{6}$ to $b_6$ mod decomposables.
    It follows from Equation \eqref{Delta-2} that the element $\Delta$ already
    exists in the $0$-line of the Adams-Novikov spectral sequence for $B$.
    %and is represented by $$\Delta = [b_4^3 - 27 y_{6}^2]\in \Ext(\BPP_\ast
    %B).$$
    This finishes the proof of Proposition \ref{Delta}.
\end{proof}

Since the map $B\to \tmf$ is an equivalence in dimensions $\leq 12$ (Corollary
\ref{12-equiv}), the elements $c_4$ and $2c_6$ lift to $\pi_\ast B$. We claim
that  $c_4 \Delta^k$ and $2c_6 \Delta^k$ live in $\pi_\ast B$; to show this, we
argue as in Remark \ref{c4-2c6}. There is a map $B\to B\wedge DA_1 \simeq T(2)$
(see also Remark \ref{B-wood}), and there is a particular complex orientation of
$\tmf_1(3)$ exhibiting it as a form of $\BP{2}$, which sits in a commutative
diagram
$$\xymatrix{
    B \ar[r] \ar[d] & T(2) \ar[d] \ar[r] & \BPP \ar[dl]\\
    \tmf \ar[r] & \tmf_1(3). &
}$$
There are choices of indecomposables $v_1$ and $v_2$ producing an isomorphism
$\pi_\ast \tmf_1(3) \cong \Z_2[v_1, v_2]$ such that $c_4$ is sent to $v_1^4$ and
$\Delta$ is sent to $v_2^4$. The map $T(2) \to \tmf_1(3)$ is surjective on
homotopy, since $v_1$ and $v_2$ live in $\pi_\ast T(2)$. Since the elements
$c_4$, $2c_6$, $c_4 \Delta^k$, and $2c_6 \Delta^k$ for $k\geq 1$ therefore
already live in the homotopy of $T(2)$, we find by the same argument that these
elements already live in the homotopy of $B$.

We next turn to showing that the other elements of $\pi_\ast \tmf$ lift to
$\pi_\ast B$. The following is the $2$-local analogue of Theorem \ref{d5-B}:
\begin{theorem}\label{d5-B-2}
    There are differentials $d_5(\Delta) = \ol{\kappa}\nu$ and $d_7(4\Delta) =
    \ol{\kappa} \eta^3$ in the ANSS for $B$. Moreover, $\ol{\kappa}\nu = 0$ in
    $\pi_\ast B$, and $8\Delta$ is a permanent cycle.
\end{theorem}
\begin{proof}
    To prove the differentials, first note that the $d_7$-differential follows
    from the $d_5$-differential via the spherical relation $4\nu = \eta^3$; it
    therefore suffices to prove the $d_5$-differential. The class
    $\ol{\kappa}\nu$ lives in bidegree $(23,5)$ in the ANSS for $B$, since it
    lives in that bidegree in the ANSS for both the sphere and for $\tmf$.
    We claim that if $\ol{\kappa} \nu^2$ vanishes in $\pi_\ast B$, then the
    $d_5$-differential follows. It suffices to establish that $d_5(\Delta\nu) =
    \ol{\kappa} \nu^2$, since the desired $d_5$-differential then follows from
    $\nu$-linearity. Since $\ol{\kappa}\nu^2$ is the first element of filtration
    $5$ in the ANSS for the sphere which does not come from an $\eta$-tower on
    the $\alpha$-family elements (and such $\eta$-towers are truncated by ANSS
    $d_3$-differentials), there cannot be any differential off it.  Moreover, if
    it is killed on any finite page in the ANSS, then it must in fact be zero in
    homotopy, since multiplicative extensions have to jump in filtration (and
    there is nothing of higher filtration). We need to show that
    $\ol{\kappa}\nu^2$ cannot be the target of a $d_r$-differential for $2\leq
    r\leq 4$; then the claimed $d_5$-differential on the $E_5$-page is forced by
    the same differential in the ANSS for $\tmf$. The algebraic
    Atiyah-Hirzebruch spectral sequence for the ANSS of $B$ implies that the
    only possibility for a differential is a $d_3$; but the source of any
    nontrivial $d_3$-differential vanishes when mapped to the ANSS for $\tmf$,
    so no such $d_3$-differential can exist.

    We now show that $\ol{\kappa} \nu^2$ vanishes in $\pi_\ast B$. For this, we
    argue as in \cite[Proposition 8.1]{hopkins-mahowald-eo2}. Namely,
    \cite[Lemma 8.2]{hopkins-mahowald-eo2} states that $\ol{\kappa} \nu^2 \in
    \langle \eta_4 \sigma, \eta, 2\rangle$. Recall that $\eta_4 = h_1 h_4$; by
    \cite[Table 21]{more-stable}, there is a $\sigma$-extension from $h_1 h_4$
    to $h_4 c_0$.  There is no indeterminacy in the above Toda bracket, so
    $\ol{\kappa} \nu^2$ will vanish if we show that $h_4 c_0$ vanishes in
    $\pi_{23}(B)$. In fact, it vanishes in the $E_2$-page of the ANSS for $B$:
    since the attaching map of the $8$-cell of $B$ is $\sigma$, the
    $\sigma$-extension on $\eta_4$ implies that $h_4 c_0$ is killed in the
    algebraic Atiyah-Hirzebruch spectral sequence for the ANSS of $B$ by a
    $d_1$-differential off the ANSS class $h_1 h_4$ supported on the cell in
    dimension $8$.
\end{proof}
Finally:
\begin{proof}[Proof of Theorem \ref{main-thm-prime-2}]
    Theorem \ref{d5-B-2} implies that $8\Delta$ lifts to $\pi_\ast B$, and that
    all the brackets in $\pi_\ast \tmf$ in Propositions \ref{toda-1},
    \ref{toda-2}, and \ref{toda-3} are well-defined in $\pi_\ast B$. The
    elements of $\pi_\ast \tmf$ in those propositions for which the bracket has
    no indeterminacy therefore lift to $\pi_\ast B$. By Remark \ref{products},
    all that remains is to show that the constant multiples of the powers of
    $\Delta$ which live in $\pi_\ast \tmf$ in fact lift to $\pi_\ast B$. Theorem
    \ref{d5-B-2} implies that they lift up to indeterminacy, and this
    indeterminacy is specified in Proposition \ref{toda-2}. If $\Delta^{8k} +
    2n\Delta^{8k} = (2n+1) \Delta^{8k}$ lifts for some $n\in \Z_{(2)}$, then so
    does $\Delta^{8k}$ since $2n+1$ is a $2$-local unit. Similarly, one finds
    that $2\Delta^{8k+4}$, $4\Delta^{4k+2}$, and $8\Delta^{2k+1}$ also lift to
    $\pi_\ast B$, as
    desired.
\end{proof}
\begin{remark}\label{ass-B}
    We briefly look at the Adams spectral sequence for $B$. The Steenrod module
    structure of the $20$-skeleton of $B$ is as in Figure \ref{B-cell}; since we
    are at the prime $2$, straight lines are $\Sq^4$, and curved lines
    correspond to $\Sq^8$ and $\Sq^{16}$, in order of increasing length. Using
    this, we can calculate the Adams spectral sequence in small dimensions. The
    Adams charts below were created with Hood Chatham's Ext calculator, and the
    Steenrod module file for $B$ in this range can be found at
    \url{http://www.mit.edu/~sanathd/input-B-leq-24-prime-2}.

    The $E_2$-page for $B$ in the first few dimensions is shown in Figure
    \ref{B-ass-2}; there are no classes in higher Adams filtration in stem $23$.
    The red class is $g = \ol{\kappa}$, and the purple lines are
    $d_2$-differentials. The differential on the class in stem $23$ already
    exists in the Adams spectral sequence for the sphere as $d_2(i) = h_0 Pd_0$.
    The other classes in stem $23$ except for the one in filtration $9$ are
    permanent cycles, and there is no multiplicative extension causing any of
    them to be $\ol{\kappa}\nu$ on homotopy.

    As shown in Figure \ref{B-ass-2-E3}, there is also a $d_3$-differential on
    the leftmost class $x_{24,1}^{(0)}$ in bidegree $(24,6)$ (which supports a
    $h_0$-tower) to the class in bidegree $(23,9)$; the class $h_0
    x_{24,1}^{(0)}$ is a permanent cycle in the ASS for $B$ which is sent to
    $8\Delta$ in the ASS for $\tmf$. The class in bidegree $(25,5)$ is a
    permanent cycle in the ASS for $B$ which is sent to $\Delta\eta$ in the ASS
    for $\tmf$.
\end{remark}

\begin{remark}\label{connections}
    We now compare the approach of this paper with that of
    \cite{hopkins-mahowald-orientations}, where the $\E{1}$-ring $B$ was
    constructed under the name $\ol{X}$. The special case of our Theorem
    \ref{main-thm} for elements in $\pi_\ast \tmf$ of ANSS filtration $0$ is
    stated as \cite[Theorem 11.1]{hopkins-mahowald-orientations}, where a proof
    is only sketched.
    
    First, their Proposition 11.2 is a combination of our Theorem \ref{d5-B} and
    Theorem \ref{d5-B-2}. Secondly, their proof proceeds by calculating the mod
    $2$ Adams spectral sequence of $B$ in dimensions $\leq 24$ to show that
    $\ol{\kappa} \nu$ vanishes in the $2$-local homotopy of $B$. Their argument
    does not seem to resolve potential multiplicative extensions: as Figure
    \ref{B-ass-2-E3} shows, there are two possibilities for multiplicative
    extensions in the Adams spectral sequence which could make $\ol{\kappa}\nu$
    nonzero in $\pi_\ast B_{(2)}$.  (Namely, the classes in bidegrees $(23,6)$
    and $(23,7)$ could represent $\ol{\kappa}\nu$.) Thirdly, Remark \ref{james}
    essentially gives a proof of their Lemma 11.5, which seems to appear without
    proof.
\end{remark}

\begin{figure}
    \includegraphics[scale=0.375]{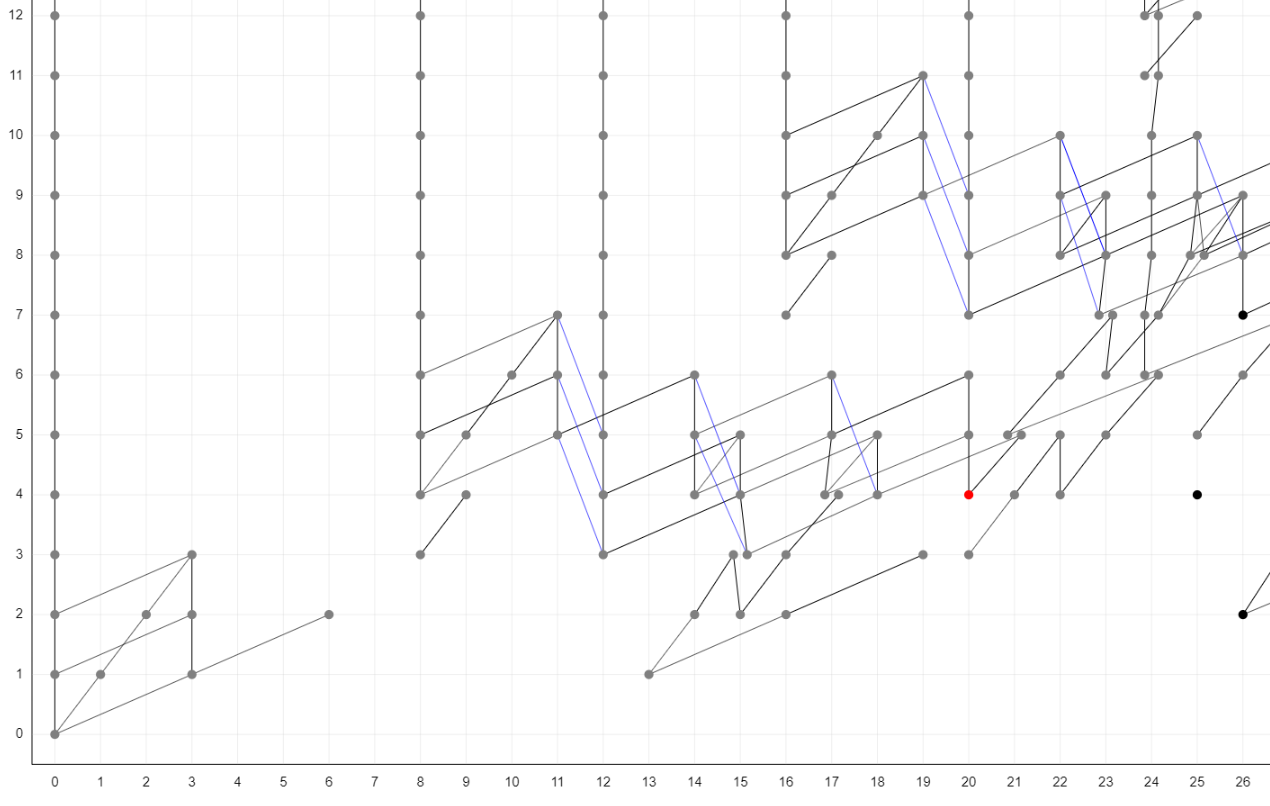}
    \caption{$E_2$-page of the Adams spectral sequence for $B$. The class
    highlighted in red is $\ol{\kappa}$.}
    \label{B-ass-2}
\end{figure}

\begin{figure}
    \includegraphics[scale=0.375]{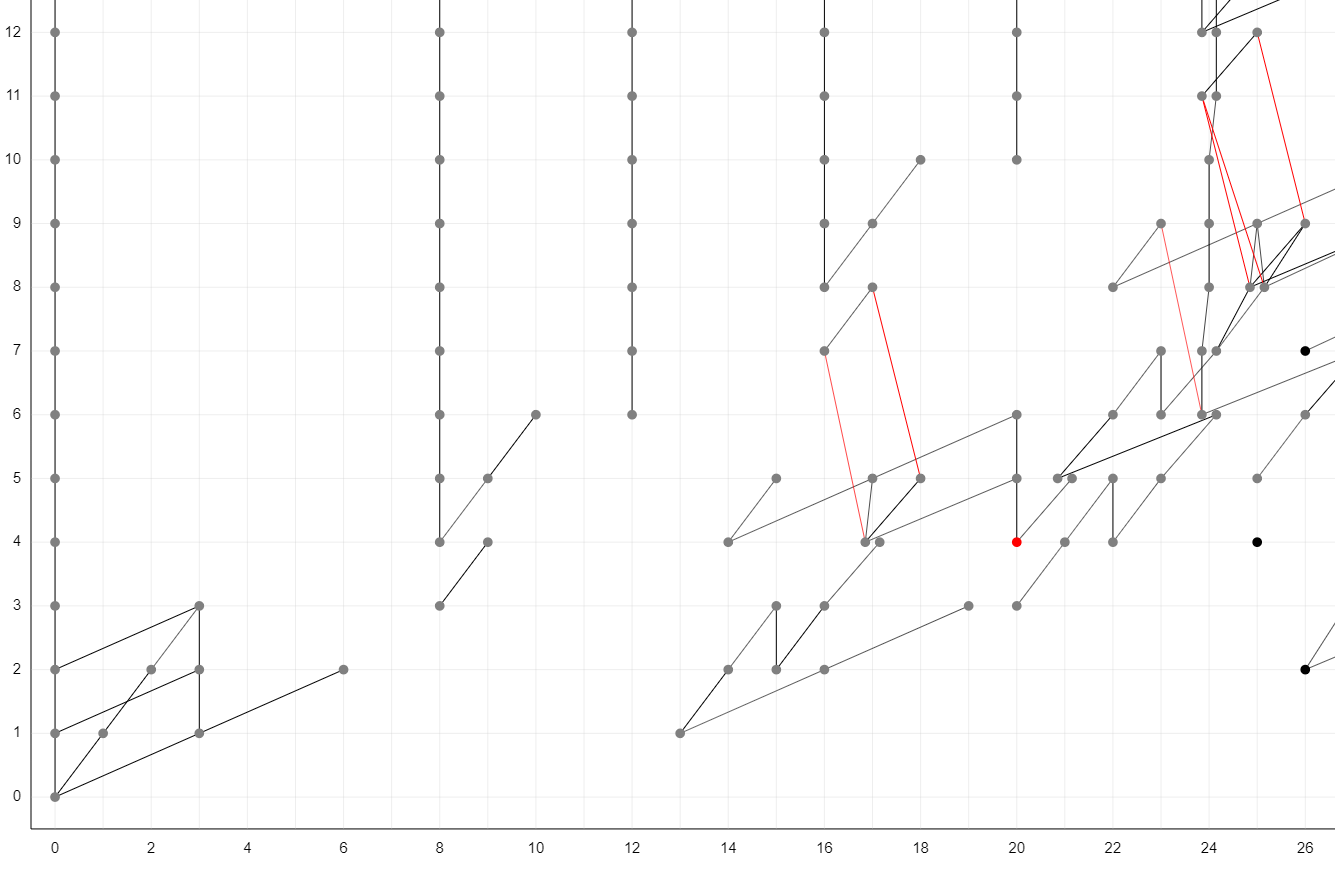}
    \caption{$E_3$-page of the Adams spectral sequence for $B$. The class
    highlighted in red is $\ol{\kappa}$. There are no differentials in this
    range from the $E_4$-page onwards.}
    \label{B-ass-2-E3}
\end{figure}

\section{Applications}\label{apps}

In this section, we study some applications of Theorem \ref{string-surj} and
Theorem \ref{main-thm}.

\subsection{A conjecture of Baker's}

In \cite{baker-conjecture}, Baker constructed a certain collection of
$\Eoo$-ring spectra $Mj_r$ with $\Eoo$-ring maps $Mj_1\to \H\Z$, $Mj_2\to \bo$,
and $Mj_3\to \tmf$. He conjectured in \cite[Conjecture 6.2]{baker-conjecture}
that the map $\pi_\ast Mj_3\to \pi_\ast \tmf$ is surjective. In this section, we
show that this conjecture follows from Theorem \ref{main-thm}.

We begin by recalling the definition of the $\Eoo$-ring spectra $Mj_r$. 
\begin{definition}
    Let $\BO\langle 2^r\rangle^{[2^{r+1}-1]}$ denote the $(2^{r+1}-1)$-skeleton
    of $\BO\langle 2^r\rangle$. Since $\BO\langle 2^r\rangle$ is an infinite
    loop space, the skeletal inclusion $\BO\langle 2^r\rangle^{[2^{r+1}-1]}\to
    \BO\langle 2^r\rangle$ induces a map $\Omega^\infty \Sigma^\infty \BO\langle
    2^r\rangle^{[2^{r+1}-1]}\to \BO\langle 2^r\rangle$. The Thom spectrum of
    this map is the $\Eoo$-ring $Mj_r$.
\end{definition}
There is an evident $\Eoo$-map $Mj_r\to \mathrm{MO}\langle 2^r\rangle$, which in
the case $r = 3$ defines an $\Eoo$-map $Mj_3\to \MString$. The following result
proves the aforementioned conjecture of Baker's as an application of Theorem
\ref{main-thm}:
\begin{prop}[{\cite[Conjecture 6.2]{baker-conjecture}}]
    The composite $Mj_3\to \MString \to \tmf$ is surjective on homotopy, where
    $\MString\to \tmf$ is the Ando-Hopkins-Rezk orientation.
\end{prop}
\begin{proof}
    By Theorem \ref{main-thm}, it suffices to show that the map $B\to \MString$
    factors through a map $B\to Mj_3$. Since $Mj_3$ is the Thom spectrum of a
    bundle over $\Omega^\infty \Sigma^\infty \BString^{[15]}$, this in turn
    follows from the existence of a map $N\to \Omega^\infty \Sigma^\infty
    \BString^{[15]}$ factoring $N\to \BString$. Recall that the map $N\to
    \BString$ was constructed via the map \eqref{fiber-sequence-map} of fiber
    sequences. The map $S^9\to \mathrm{B^2 String}$ factors as $S^9\to
    \Omega^{\infty-1} \Sigma^\infty \BString^{[15]}$, and so the map of fiber
    sequences in \eqref{fiber-sequence-map} factors as
    \begin{equation*}
	\xymatrix{
	    N \ar[r] \ar[d]_-\exists & \Omega S^{13} \ar[r] \ar[d] & S^9
	    \ar[d]\\
	    \Omega^\infty \Sigma^\infty \BString^{[15]} \ar[r] \ar[d] & \ast
	    \ar[r] \ar[d] & \Omega^{\infty-1} \Sigma^\infty \BString^{[15]}
	    \ar[d]\\
	    \BString \ar[r] & \ast \ar[r] & \mathrm{B^2 String},
	    }
    \end{equation*}
    as desired.
\end{proof}

\subsection{Hirzebruch's prize question}

Another application of Theorem \ref{string-surj} was stated as \cite[Corollary
6.26]{hopkins-icm}, and provides an answer to Hirzebruch's prize question
\cite[Page 86]{hirzebruch}. See also \cite{hopkins-mahowald-orientations}.
\begin{corollary}\label{prize}
    There exists a $24$-dimensional compact smooth string manifold $M$ with
    $\hat{A}(M)=1$ and $\hat{A}(M,\tau_M\otimes \cc)=0$.
\end{corollary}
\begin{proof}
    By the discussion on \cite[Page 86]{hirzebruch}, the conditions on the
    $\hat{A}$-genus of $M$ are equivalent to the Witten genus of $M$ being
    $c_4^3 - 744 \Delta = \Delta(j - 744)$, where $j$ is the $j$-function. Let
    $M^8_0$ denote the Kervaire-Milnor almost parallelizable $8$-manifold; then,
    the $8$-manifold $-M^8_0 - 224 \HP^2$ (whose string cobordism class we will
    denote by $[N_{c_4}]$, where $N_{c_4}$ is the explicit manifold
    representative above) admits a string structure by \cite[Lemma
    15]{laures-k1-local}. The map $\tmf \to \bo$ sends $c_4 \in \pi_8 \tmf$ to
    $v_1^4\in \pi_8 \bo$. By Lemma \ref{commute-lemma}, there is a commutative
    diagram:
    \begin{equation}\label{string-commute}
	\xymatrix{
	\MString \ar[r] \ar[d] & \MSpin \ar[d]\\
	\tmf \ar[r] & \bo,}
    \end{equation}
    where the left vertical map is the Ando-Hopkins-Rezk orientation and the
    right vertical map is the Atiyah-Bott-Shapiro orientation. Consequently, the
    Witten genus of $-M^8_0 - 224 \HP^2$ is $c_4$.

    By Theorem \ref{string-surj}, the element $24\Delta\in \pi_{24} \tmf$ lifts
    to a class $[N_\Delta]$ in $\pi_{24} \MString$, where $N_\Delta$ is any
    manifold representative. Since $744\Delta = 31\cdot 24\Delta$, we conclude
    that the string cobordism class of the $24$-dimensional compact oriented
    smooth string manifold $N_{c_4}^3 - 31 N_\Delta$ has Witten genus $c_4^3 -
    744 \Delta$, as desired.
\end{proof}
The proof of Corollary \ref{prize} utilized the following lemma.
\begin{lemma}\label{commute-lemma}
    The diagram \eqref{string-commute} commutes, where the left vertical map is
    the Ando-Hopkins-Rezk orientation and the right vertical map is the
    Atiyah-Bott-Shapiro orientation. 
\end{lemma}
\begin{proof}
    We need to show that the composite $\MString\to \tmf \to \bo$ comes from the
    Atiyah-Bott-Shapiro orientation. By \cite[Corollary 7.12]{koandtmf}, it
    suffices to show that this composite has the same characteristic series as
    the restriction of the $\hat{A}$-genus to string manifolds. There is an
    isomorphism $\pi_\ast \bo\otimes \QQ \cong \Z[\beta^2]$, where $\beta^2$
    lives in degree $4$ and is the square of the Bott element.  Moreover,
    $\pi_\ast \tmf \otimes \QQ$ is isomorphic to the ring of rational modular
    forms (of weight given by half the degree in $\pi_\ast \tmf\otimes \QQ$) by
    \cite[Proposition 4.4]{bauer-tmf}. The map $\pi_\ast \tmf \otimes \QQ\to
    \pi_\ast \bo\otimes \QQ$ sends a modular form of weight $k$ with
    $q$-expansion $f(q) = \sum a_n q^n$ to the element $a_0 (\beta^2)^k\in
    \pi_{2k}\bo\otimes \QQ$. Consequently, the composite $\pi_\ast\MString\to
    \pi_\ast \tmf \otimes \QQ \to \pi_\ast \bo\otimes \QQ$ sends a string
    manifold $M$ to the constant term of the $q$-expansion of its Witten genus.
    The lemma will therefore follow if this constant term is the $\hat{A}$-genus
    of $M$, but this follows from the discussion on \cite[Page 84]{hirzebruch}.
\end{proof}
\begin{remark}
    The modular form $c_4^3 - 744 \Delta$ is $\theta_{\Lambda_{24}} - 24
    \Delta$, where $\Lambda_{24}$ is the $24$-dimensional Leech lattice and
    $\theta_{\Lambda_{24}}$ is its theta function.
\end{remark}
\begin{remark}\label{monster-action}
    The original motivation for Hirzebruch's prize question was to relate the
    geometry of the $24$-dimensional string manifold $M$ of Corollary
    \ref{prize} to representations of the monster group by constructing an
    action of the monster group on $M$. The question of constructing this action
    remains unresolved.
\end{remark}
\begin{remark}
    The disussion on \cite[Page 86]{hirzebruch} implies that $\hat{A}(N_\Delta)
    = 0$ and $\hat{A}(N_\Delta, \tau_{N_\Delta}\otimes \cc) = 24$. It follows
    from \cite[Theorem A]{stolz-scalar} that $N_\Delta$ (which we may assume is
    simply-connected by surgery) admits a metric with positive scalar curvature.
    Since the Witten genus of $N_\Delta$ is nonzero, Stolz's conjecture in
    \cite{stolz-ricci} would imply that it does not admit a metric of
    positive-definite Ricci curvature. We do not know whether Stolz's conjecture
    holds in this particular case. Note, however, that there are examples of
    non-simply-connected manifolds which admit positive scalar curvature metrics
    but no metrics of positive-definite Ricci curvature: as pointed out to us by
    Stolz, a connected sum of lens spaces of dimension at least $3$ gives such a
    manifold.
\end{remark}

Corollary \ref{prize} may be generalized in the following manner. Recall the
following definition from \cite[Section 2.3]{ono-web}. Let $j_1(z) = j(z) -
744$, and define $j_n(z)$ for $n\geq 2$ via $nT_n(j_1(z))$, where $T_n$ is the
weight zero Hecke operator, acting on $f(z)$ via
$$T_n f(z) = \sum_{d|n, ad = n} \sum_{b=0}^{d-1} f\left(\frac{az+b}{d}\right).$$
By \cite[Proposition 2.13]{ono-web}, $j_n(z)$ is a monic integral polynomial in
$j(z)$ of degree $n$; for instance,
$$j_2(z) = j(z)^2 - 1488 j(z) + 159768, \ j_3(z) = j(z)^3 - 2232 j(z)^2 +
1069956 j(z) - 36866976.$$
The functions $j_n(z)$ for $n\geq 0$ (where $j_0(z) = 1$) form a basis for the
complex vector space of weakly holomorphic modular forms of weight $0$, and
appear in the denominator formula for the monster Lie algebra. They may be
defined by Faber polynomials on $j$. The generalization of Corollary \ref{prize}
is as follows.
\begin{prop}\label{generalization}
    For all $n\geq 0$, there is a $24n$-dimensional compact smooth string
    manifold $M^{24n}$ whose Witten genus is $\Delta^n j_n(z)$.
\end{prop}
\begin{remark}
    By arguing as in \cite[Pages 86-87]{hirzebruch}, we find that the twisted
    $\hat{A}$-genera of bundles over $M^{24n}$ constructed from the complexified
    tangent bundle of $M$ are integral linear combinations of dimensions of
    irreducible representations of the monster group; for instance,
    $\hat{A}(M^{48}; \Sym^2(\tau_M\otimes \cc))$ is the coefficient of $q^2$ in
    $\Delta^2 j_2(z)$, which is $2\times (21296876 + 196883 + 1)$. More
    generally, $\hat{A}(M^{24n}; \Sym^2(\tau_M\otimes \cc))$ is an integral
    linear combination of the dimensions of the $n$ smallest irreducible
    representations of the monster group. In light of Hirzebruch's original
    motivation for his prize question (see Remark \ref{monster-action}), it
    seems reasonable to conjecture that the $24n$-dimensional string manifold
    $M^{24n}$ admits an action of the monster group by diffeomorphisms.
\end{remark}
\begin{remark}
    It would be interesting to know if there is an analogue of Proposition
    \ref{generalization} for other McKay-Thompson series.
\end{remark}
Before providing the proof, we need the following result.
\begin{theorem}\label{tmf-htpy}
    A modular form $f$ is in the image of the boundary homomorphism $\pi_\ast
    \tmf \to \mathrm{MF}_\ast$ in the Adams-Novikov spectral sequence if and
    only if it is expressible as an integral linear combination of monomials of
    the form $a_{ijk} c_4^i c_6^j \Delta^k$ with $i,k\geq 0$ and $j=0,1$, where
    $$a_{ijk} = \begin{cases}
	1 & i>0,j=0\\
	2 & j=1\\
	24/\gcd(24,k) & i,j=0.
    \end{cases}$$
\end{theorem}
\begin{proof}
    This is \cite[Proposition 4.6]{hopkins-icm}, proved in \cite{bauer-tmf}.
\end{proof}
\begin{proof}[Proof of Proposition \ref{generalization}]
    We have
    $$\Delta^n j_n(z) = \sum_{0\leq k\leq n} \alpha_k j(z)^k \Delta^n =
    \sum_{0\leq k\leq n} \alpha_k c_4^{3k} \Delta^{n-k},$$
    for some integers $\alpha_k$ (where $\alpha_n = 1$). By Theorem
    \ref{string-surj} and Theorem \ref{tmf-htpy}, it suffices to show that the
    constant term $\alpha_0$ of $j_n(z)$ (when expanded as a monic integral
    polynomial in $j(z)$) is a multiple of $24/\gcd(24,n)$. The $j$-function
    vanishes on a primitive third root of unity, so $\alpha_0 = j_n(\omega)$.
    Its generating function is
    $$\sum_{n\geq 0} j_n(\omega) q^n = -\frac{j'(z)}{j(z)} = \frac{c_6}{c_4},$$
    where $q = e^{2\pi i z}$ and $\omega$ is a primitive third root of unity.

    Let $m\geq 1$; we claim that the coefficients $a_{4,m}$ and $a_{6,m}$ of
    $q^m$ in the $q$-expansion for $c_4$ and $c_6$ (respectively) are divisible
    by $24/\gcd(24,m)$. Indeed, the expression for their $q$-expansion shows
    that $a_{4,n} = -240\sigma_3(n)$ and $a_{6,n} = 504\sigma_5(n)$, and both
    $240$ and $504$ are already divisible by $24$. Since the coefficient of
    $q^m$ in $1/c_4$ can be expressed as an integral linear combination of the
    $a_{4,k}$, it follows that the coefficient of $q^m$ for $m\geq 1$ in
    $c_6/c_4$ (which is $j_m(\omega)$) is divisible by $24$, and hence by
    $24/\gcd(24,m)$, as desired.
\end{proof}

\bibliographystyle{alpha}
\bibliography{main}
\end{document}